\documentclass[a4paper,11pt,hidelinks]{amsart}
\usepackage{graphicx} 
\usepackage[a4paper, margin=2 cm]{geometry}
\usepackage{hyperref}
\usepackage{amsthm}
\usepackage{thm-restate}
\usepackage{geometry}
\usepackage{amssymb}
\usepackage{amsmath}

\usepackage{cleveref}
\usepackage{comment}
\usepackage{setspace}
\usepackage{xcolor}
\usepackage[shortlabels]{enumitem}
\usepackage{listings}
\usepackage{blkarray}
\usepackage{lmodern}
\usepackage[english]{babel}
\usepackage{float}
\usepackage{tikz}
\usetikzlibrary{arrows.meta,arrows}
\usetikzlibrary{arrows,decorations.markings}
\usetikzlibrary{decorations.pathmorphing}
\usepackage[square,sort,comma,numbers]{natbib}

\makeatletter
\def\thm@space@setup{\thm@preskip=5pt
\thm@postskip=5pt}
\makeatother

\newtheorem{theorem}{Theorem}[section]
\newtheorem{lemma}[theorem]{Lemma}

\newtheorem*{claim*}{Claim}
\newtheorem{proposition}[theorem]{Proposition}
\newtheorem{conjecture}[theorem]{Conjecture}
\newtheorem{remark}[theorem]{Remark}

\theoremstyle{definition}

\newtheorem*{definition*}{Definition}

\newcommand{\eps}{\varepsilon}

\title{Cyclic subsets in regular Dirac graphs}

\author
{
Nemanja Dragani\'c
}
\author
{
Peter Keevash
}
\author
{
Alp Müyesser 
}

\address[Dragani\'c]{Mathematical Institute, University of Oxford, UK. 
Supported by SNSF project 217926.}

\email{\tt nemanja.draganic@maths.ox.ac.uk}

\address[Keevash]{Mathematical Institute, University of Oxford, UK. 
Supported by ERC Advanced Grant 883810.}

\email{\tt keevash@maths.ox.ac.uk}

\address[Müyesser]{New College, University of Oxford, UK.}

\email{\tt alp.muyesser@new.ox.ac.uk}

\date{}

\newcommand{\cyc}{\operatorname{Cyc}}

\newcommand{\Gh}{G[\tfrac12]}

\newcommand{\mc}[1]{\mathcal{#1}}
\newcommand{\mb}[1]{\mathbb{#1}}
\newcommand{\nib}[1]{\noindent {\bf #1}}
\newcommand{\nim}[1]{\noindent {\em #1}}

\newcommand{\bcl}[1]{\left\lceil #1 \right\rceil}

\newcommand{\sub}{\subseteq}

\newcommand{\sm}{\setminus}
\newcommand{\ov}{\overline}

\newcommand{\wt}{\widetilde}

\newcommand{\aA}{\alpha}
\newcommand{\bB}{\beta}
\newcommand{\gG}{\gamma}
\newcommand{\dD}{\delta}

\newcommand{\lL}{\lambda}
\newcommand{\tT}{\theta}

\newcommand{\OO}{\Omega}

\newcommand{\DD}{\Delta}

\begin{document}

\begin{abstract}
In 1996, in his last paper, Erd\H{o}s asked the following question that he formulated together with Faudree: 
is there a positive $c$ such that any $(n+1)$-regular graph $G$ on $2n$ vertices 
contains at least $c 2^{2n}$ distinct vertex-subsets $S$ that are cyclic,
meaning that there is a cycle in $G$ using precisely the vertices in $S$.
We answer this question in the affirmative in a strong form by proving 
the following exact result: if $n$ is sufficiently large and $G$ minimises the number of cyclic subsets
then $G$ is obtained from the complete bipartite graph $K_{n-1,n+1}$
by adding a $2$-factor (a spanning collection of vertex-disjoint cycles) within the part of size $n+1$.
In particular, for $n$ large, this implies that the optimal $c$ in the problem is precisely $1/2$. 
 \end{abstract}

\maketitle

\section{Introduction}

An important recent theme in graph theory has been the study of when certain classical theorems
hold in a `robust' or `resilient' way, according to various possible interpretations of these terms.
While many classical results can be interpreted as part of this direction of research,
it was first highlighted as a topic for systematic study by Sudakov and Vu \cite{sudakov2008local}.
The survey by Sudakov  \cite{sudakov2017robustness} discusses many types of robustness,
primarily illustrated by the Hamiltonicity problem, which will also be our focus in this paper.
Here the fundamental theorem is that of Dirac \cite{dirac1952}, stating that any  graph $G$ on $n$ vertices
with minimum degree $\dD(G) \ge n/2$ (commonly referred to as a \emph{Dirac graph}) 
contains a cycle that is Hamiltonian (meaning that it uses all vertices of $G$). 

There are many further results showing that Dirac graphs are Hamiltonian in a deeper and more robust sense,
in strong contrast to graphs of lower minimum degree, which may not even be connected.
These stronger theorems include showing that Dirac graphs have many Hamiltonian cycles \cite{SSS2003, cuckler2009hamiltonian}
and remain Hamiltonian even after passing to a $p$-random subgraph with $p \ge \OO(\tfrac{\log n}{n})$, see  \cite{krivelevich2014robust}.
Further results on resilience can be found in \cite{montgomery2019hamiltonicity, ben2011resilience}
and on pancyclicity in \cite{draganic2024pancyclicity, letzter2023pancyclicity, krivelevich2010resilient, bondy1971pancyclic}.
Another type of robustness considers decomposing $G$ into Hamiltonian cycles  \cite{kuhn2013hamilton, csaba2016proof} 
or covering $G$ by Hamiltonian cycles \cite{ferber20181, draganic2025optimal, krivelevich2012optimal, knox2015edge, hefetz2014optimal}.

Here we consider another natural measure of robustness: are many induced subgraphs Hamiltonian?
This problem was posed by Erd\H{o}s and Faudree \cite{erdHos1999selection} (see also \cite[Problem 622]{erdosproblems}),
who asked whether in any regular Dirac graph, a constant fraction of vertex subsets are \emph{cyclic}, 
meaning that they induce a Hamiltonian subgraph. We let $\cyc(G)$ count cyclic subsets in $G$
and formulate the statement as a conjecture (they ask `Is it true?). 

\begin{conjecture} \label{conj:EF}
Any  $(n+1)$-regular graph $G$ on $2n$ vertices 
has $\cyc(G) \ge c 2^{2n}$ for some absolute $c>0$.
\end{conjecture}

We make the following remarks on this conjecture:
\begin{enumerate}
\item One cannot replace $n+1$ by $n$, as if $G=K_{n,n}$ is a balanced complete bipartite graph
then almost all subsets of $V(G)$ induce an unbalanced complete bipartite graph, which is clearly not Hamiltonian.
\item The regularity assumption cannot be weakened to minimum degree $n+1$, as can be seen by considering $K_{n,n}$
with two spanning stars added within each part.
\item As Erd\H{o}s notes in \cite{erdHos1999selection}, one cannot take $\eps>1/2$,
due to the examples $\mc{G}_n$ defined below.
\end{enumerate}

Let $\mc{G}_n$ be the set of $(n+1)$-regular graphs $G$ on $2n$ vertices obtained from the complete bipartite graph $K_{n-1,n+1}$
by adding a $2$-factor (a spanning collection of vertex-disjoint cycles) within the part of size $n+1$.

Our main result is that for large $n$ the extremal examples for \Cref{conj:EF} belong to $\mc{G}_n$.

\begin{theorem}\label{thm:main+}
For any $(n+1)$-regular graph $G$ on $2n$ vertices with $n$ sufficiently large
we have $\cyc(G) \ge \min_{G' \in \mc{G}_n} \cyc(G')$.
\end{theorem}

In particular, this shows that \Cref{conj:EF} holds with $c=1/2$ for large $n$,
as for any $G' \in \mc{G}_n$ the proportion of subsets that are cyclic is
$\tfrac12 + \tfrac{3/2}{\sqrt{n\pi}} + O(n^{-3/2})$ (see \Cref{lem:pn} below).

Remark (2) above suggests the question of what minimum degree is needed 
in a graph on $n$ vertices to ensure $\cyc(G) = \Omega(2^n)$. 
A similar construction with more stars shows that
minimum degree $n/2+o(\sqrt{n})$ is not sufficient. On the other hand, a well-known degree sequence generalisation of Dirac's theorem due to Chv\'atal \cite{chvatal1972hamilton} implies that a minimum degree of $n/2+\Omega(\sqrt{n})$ is sufficient.

\begin{proposition} \label{thm:mindeg}
Any graph $G$ on $n$ vertices with minimum degree $\dD(G) \ge n/2+\Omega(\sqrt{n})$
has  $\cyc(G) \ge \Omega(2^n)$.
\end{proposition}

To deduce \Cref{thm:mindeg} from \cite{chvatal1972hamilton} it suffices to observe via Chernoff bounds 
that for $G$ as in the statement, a positive fraction of all induced subgraphs $G'$ of $G$ 
have minimum degree at least $0.49|V(G')|$ and at least $|V(G')|/2$ vertices of degree at least $|V(G')|/2$;
we omit further details.

\medskip

\nib{Notation.}
Let $G$ be a graph.
We denote its vertex set by $V(G)$, its edge set by $E(G)$,
its minimum degree by $\delta(G)$ and its maximum degree by $\Delta(G)$.
For $A \sub V(G)$ we write $\ov{A} = V(G) \sm A$.
For $v \in V(G)$ we let $d_G(v,A)$ count neighbours of $v$ in $A$
and $\bar{d}(v,A)=|A|-d_G(v,A)$ count non-neighbours of $v$ in $A$. We denote by $e(G)$ the number of edges in $G$.

For $A,B \sub V(G)$ we write $G[A]$ for the subgraph of $G$ induced by $A$. 
We let $e(A)=e(G[A])$ and $e(A,B)=\{(a,b)\in E(G)\mid a\in A, b\in B\}$, where edges in $A\cap B$ are counted twice. 
If $A$ and $B$ are disjoint, we denote by $G[A,B]$ the bipartite subgraph of $G$ induced by $(A,B)$.
We also use $\bar{e}$ for non-edges, e.g.~writing $\bar{e}(A,B)=|A||B|-e(A,B)$ for disjoint $A,B$.

Let $G$ be a graph on $n$ vertices.
If $A \sub V(G)$ with $|A|\in\{\lfloor{n/2}\rfloor,\lceil n/2\rceil\}$ we call $A$ a \emph{half-set}.
A cut of $G$ is a partition of $V(G)$ into two parts.
If $(A,B)$ is a cut of $G$ where $A,B$ are half-sets, we call $(A,B)$ a \emph{balanced cut}.

We write $a \ll b$ if for any $b>0$ there is $a>0$ such that the following statement holds;
hierarchies with more parameters are defined similarly.
We write $a = o_t(1)$ if $a \to 0$ as $t \to \infty$; if $t$ is not specified then $n$ is understood.
Our asymptotic notation $O(\cdot)$, $\Omega(\cdot)$, etc. refers to large $n$.
We say that an event $E$ holds with high probability (whp) if $\mb{P}(E)=1-o(1)$.

\medskip

\nib{Organisation.}
We start with an overview in \Cref{sec:overview},
where we prove \Cref{conj:EF} modulo three lemmas.
These lemmas are proved in \Cref{sec:ham}, 
which concerns various sufficient conditions for Hamiltonicity in near-Dirac graphs.
\Cref{sec:asymptotic} contains the proof of our asymptotic result, \Cref{thm:main}, which states that \Cref{conj:EF} holds with $c=1/2-o(1)$. Extending the methods developed in previous sections, including a more delicate analysis, we prove the exact result \Cref{thm:main+} in \Cref{sec:exact}.
\Cref{sec:rem} contains some concluding remarks.

\section{Overview} \label{sec:overview}

The first ingredient of our proof is a well-known classification of Dirac graphs into three cases,
which we will call `bi-dense', `almost two cliques' and `almost bipartite'.
Here we use the following lemma of Krivelevich, Lee and Sudakov \cite{krivelevich2014robust},
which has its origins in the work of Koml\'os, Sark\"ozy and Szemer\'edi \cite{komlos1998proof} 
(see also \cite{sarkozy2008distributing, cuckler2009hamiltonian}). 
For the first case, we recall that a half-set in a graph is a set containing $\lfloor n/2 \rfloor$ or  $\lceil n/2 \rceil$ vertices.

\begin{lemma}[\cite{krivelevich2014robust}] \label{lem:casesdirac}
Let $\eps\le\frac{1}{320}$ and $\gamma \le\frac{1}{10}$ be fixed
positive reals such that $\gamma \ge 32\eps$. If $n$ is large
enough, then for every graph $G$ on $n$ vertices with minimum degree
at least $\frac{n}{2}$, one of the following cases holds:
\begin{enumerate}[(i)]
  \setlength{\itemsep}{1pt} \setlength{\parskip}{0pt}
  \setlength{\parsep}{0pt}
\item \textbf{(Bi-dense)} $e(A,B)\ge\eps n^{2}$ for all half-sets $A$ and $B$ (not necessarily disjoint),
\item \textbf{(Almost two cliques)} There exists $A \sub V(G)$ of size $\frac{n}{2}\le|A|\le(\frac{1}{2}+16\eps)n$
such that $e(A,\ov{A})\le6\eps n^{2}$ and the induced subgraphs
on both $A$ and $\ov{A}$ have minimum degree at least $\frac{n}{5}$,
or
\item \textbf{(Almost bipartite)} There exists $A \sub V(G)$ of size $\frac{n}{2}\le|A|\le(\frac{1}{2}+16\eps)n$
such that  $G[A,\ov{A}]$ has at least $(\frac{1}{4}-14\eps)n^{2}$
edges and minimum degree at least $\frac{\gamma}{2}n$.
Moreover, either $|A|=\lceil\frac{n}{2}\rceil$, or the induced
subgraph $G[A]$ has maximum degree at most $\gamma n$.
\end{enumerate}
\end{lemma}

Given this classification, the proof of \Cref{conj:EF} will be completed by the following three lemmas,
which give the required probability in each of the three cases 
that a random induced subgraph $\Gh$ is Hamiltonian,
thus proving the following probabilistic reformulation of  \Cref{conj:EF}.

\begin{theorem}\label{thm:main-} 
Any  $(n+1)$-regular graph $G$ on $2n$ vertices 
has $\mb{P}(\Gh \text{ is Hamiltonian}) > \Omega(1)$.
\end{theorem}

In the first two cases, Hamiltonicity of $\Gh$ holds whp;
it is the third case that determines the bound on the probability,
as one would expect given the extremal example.
The main difficulty of the proof will lie in the third case,
and the accuracy of our probability estimates here determine
whether we prove \Cref{thm:main-}, or the stronger asymptotic result \Cref{thm:main},
or the even stronger exact result \Cref{thm:main+}.

\begin{lemma} \label{lem:bidense}
Let $\eps>0$ and $G$ be an $(n+1)$-regular graph on $2n$ vertices. 
Suppose that $e(G[A,B])\ge\eps n^{2}$ for all half-sets $A$ and $B$.
Then whp $\Gh$ is Hamiltonian.
\end{lemma}

\begin{lemma} \label{lem:near2c}
Let $n^{-1} \ll \eps \ll  1$ and $G$ be an $(n+1)$-regular graph on $2n$ vertices. 
Suppose there is some $A \sub V(G)$ with $n \le|A|\le(1+32\eps)n$
such that $e(A,\ov{A}) \le 24\eps n^{2}$ 
and $G[A]$, $G[\ov{A}]$ both have minimum degree at least $2n/5$. 
Then whp $\Gh$ is Hamiltonian.
\end{lemma}

\begin{lemma} \label{lem:nearbip}
Let $n^{-1} \ll \eps \ll \gG \ll 1$ and $G$ be an $(n+1)$-regular graph on $2n$ vertices. 
Suppose there is some $A \sub V(G)$ with $n \le|A|\le(1+32\eps)n$
such that $G[A,\ov{A}]$ has at least $(1-56\eps)n^{2}$
edges and minimum degree at least $\gamma n$, where if  $|A|>n$ then
 $G[A]$ has maximum degree at most $\gamma n$.
Then $\mb{P}[ \Gh \text{ is Hamiltonian}] \ge 0.01$.
\end{lemma}

We prove these lemmas in the following section,
thus completing the proof of  \Cref{conj:EF}.

\section{Conditions for Hamiltonicity} \label{sec:ham}

In this section we formulate conditions for Hamiltonicity
in each of the three cases from \Cref{lem:casesdirac}.
We use these to prove the lemmas stated in the previous section,
thus completing the proof of \Cref{conj:EF}.

\subsection{Bi-dense graphs}

In this subsection we prove \Cref{lem:bidense},
which establishes \Cref{thm:main-} for bi-dense graphs.
This will follow from the following stability result for Dirac's theorem,
which is \cite[Lemma 11]{komlos1996square}.

\begin{theorem}[Stability for Dirac's theorem]\label{thm:diracstability}
For any $\eps>0$, if $n$ is large
and $G$ is a graph on $n$ vertices with  $\delta(G) \geq (\frac{1}{2}-\eps)n$
then one of the following holds:
(i) $G$ is Hamiltonian,
(ii) $e(G[A])=0$ for some $A \sub V(G)$ with $|A|=(1/2-\eps)n$, or
(iii) $e(G[A,B]) \le n$ for some disjoint $A,B \sub V(G)$ with $|A|=|B|=(1/2-\eps)n$.
\end{theorem}

We remark that this source of this result appears not to be widely known (it is sometimes considered folklore)
and that a nice short proof is given in the appendix of \cite{komlos1996square}, using the following ideas.
Supposing that $G$ is not Hamiltonian, one shows that the longest cycle $C_1$ has length $\sim n$ or $\sim n/2$,
and in the latter case there is a disjoint cycle $C_2$ of length $\sim n/2$.
In the first case, a set $A$ as in (ii) is obtained by the standard Chv\'atal-Erd\H{o}s / P\'osa argument:
consider any vertex not in $C_1$ and its shifted neighbourhood on $C_1$.
In the second case, by maximality of $C_1$ there can be at most $n$ edges from $C_1$ to $C_2$.

The idea of the proof of \Cref{lem:bidense} is that whp $\Gh$ satisfies the hypotheses
of \Cref{thm:diracstability} and does not satisfy conclusion (ii) or (iii),
so satisfies conclusion (i), meaning it is Hamiltonian.
The hypotheses are easy consequences of the following well-known Chernoff bound 
(see e.g.~\cite[Theorem~2.1]{janson2011random}),
which will be used throughout the paper.

\begin{theorem}[Chernoff bound]
Let $X$ be a binomial random variable and $\delta>0$. Then 
\[\mb{P}[ |X-\mb{E}X| \ge \delta \mb{E}X ] \le 2\exp(-\delta^2\mb{E}X/(2+\delta)).\]
\end{theorem}

To consider the conclusions of \Cref{thm:diracstability} for $\Gh$,
we will need the standard fact that bidenseness is whp inherited by the random induced subgraph $\Gh$.
For completeness, we will include a short argument for this in the proof of \Cref{lem:bidense},
using the following Frieze-Kannan Regularity Lemma \cite{frieze1999quick}
(or one could use the Szemer\'edi Regularity Lemma,
but this would give much worse bounds for $n$).

\begin{lemma} \label{lem:FK}
For any $\eps>0$ there are $T,n_0$ so that for any graph $G$ on $n \ge n_0$ vertices
there is a partition $(V_1,\dots,V_t)$ of $V(G)$ with $t \le T$ and $|V_1| \le \cdots \le |V_t| \le |V_1|+1$
such that for any $A,B \sub V(G)$ we have
\[ \Big| e(G[A,B]) - \sum_{i,j} \tfrac{|A \cap V_i|}{|V_i|} \tfrac{|B \cap V_j|}{|V_j|} e(G[V_i,V_j]) \Big| < \eps n^2. \]
\end{lemma}

We conclude this subsection by proving \Cref{lem:bidense}.

\begin{proof}[Proof of \Cref{lem:bidense}]
Let $\eps>0$ and $G$ be an $(n+1)$-regular graph on $2n$ vertices. 
Suppose that $e(G[A,B])\ge\eps n^2$ for all half-sets $A$ and $B$.
Consider $\Gh = G[S]$ where $S$ is a random subset of $V(G)$.
By Chernoff bounds, whp $|S| = n \pm n^{0.6}$ and
every vertex has degree $n/2 \pm n^{0.6}$ in $\Gh$.

We intend to apply \Cref{thm:diracstability} with $(|S|,0.1\eps)$ in place of $(n,\eps)$.
The assumptions of this theorem are satisfied,
so to complete the proof, we need to show that whp $G[S]$ 
does not satisfy conclusion (ii) or (iii),
and so it must satisfy conclusion (i), meaning it is Hamiltonian.
To do so, it suffices to show that bidenseness is inherited,
in that whp for any $A',B' \sub S$ with $|A'|,|B'| \sim n/2$
we have $e(G[A',B']) \ge 0.1\eps |S|^{2}$, say.

To see this, let $(V_1,\dots,V_t)$ be a partition of $V(G)$
obtained from \Cref{lem:FK} applied with $0.1\eps$ in place of $\eps$.
By Chernoff bounds, whp each $|S \cap V_i| = \tfrac{n}{2t} \pm n^{0.6}$.
Consider  any $A',B' \sub S$ with $|A'|,|B'| \sim n/2$.
Construct $A,B \sub V(G)$ with $|A|,|B| \sim n$ so that each
$|A \cap V_i| \sim 2|A' \cap V_i|$ and $|B\cap V_i| \sim 2|B' \cap V_i|$.
Then $e(G[A,B])\ge\eps n^2 - o(n^2)$ by bidenseness of $G$.
Applying the conclusion of \Cref{lem:FK} to $(A,B)$ and $(A',B')$,
we see that $e(G[A,B])$ differs by at most $0.1\eps n^2$ from
$\sum :=  \sum_{i,j} \tfrac{|A \cap V_i|}{|V_i|} \tfrac{|B \cap V_j|}{|V_j|} e(G[V_i,V_j])$
and  $e(G[A',B'])$ differs by at most $0.1\eps n^2$ from 
$ \sum_{i,j} \tfrac{|A' \cap V_i|}{|V_i|} \tfrac{|B' \cap V_j|}{|V_j|} e(G[V_i,V_j]) \sim \sum / 4$.
Therefore $e(G[A',B']) \ge 0.1\eps |S|^{2}$.
\end{proof}


\subsection{Almost two cliques}

In this subsection we prove \Cref{lem:near2c},
which establishes \Cref{thm:main-} for graphs that are almost two cliques.
We use the following sufficient condition for Hamiltonicity in such graphs.

\begin{lemma}\label{lem:twocliqueshamilton} 
Let $G$ be a graph on $n$ vertices
and $(A,B)$ be a partition of $V(G)$ with $|A|,|B|\geq 0.49n$.
Suppose that $G[A,B]$ contains two disjoint edges
and $G[A]$, $G[B]$ both have minimum degree at least $n/100$ 
and at most $10^{-4} n^2$ non-edges.
Then $G$ is Hamiltonian.
\end{lemma}

The proof of \Cref{lem:twocliqueshamilton} 
 uses the following folklore result on Hamilton-connectivity (we include the short proof for completeness).

\begin{lemma}\label{lem:hamiltonconnected}
 Let $G$ be a graph on $n$ vertices with minimum degree $\delta(G)\geq n/2 + 1$.
 Then $G$ is Hamilton-connected, meaning for any distinct $a,b\in V(G)$, there exists a Hamilton path from $a$ to $b$.
\end{lemma}
\begin{proof}
Consider any $a,b\in V(G)$ and obtain $G'$ from $G$ by deleting $a$ and $b$.
Then $|V(G')|=n-2$ and $\delta(G')\geq |V(G')|/2$, so by Dirac's theorem $G'$ has a Hamiltonian cycle $C$.
As $a$ and $b$ each have at least $n/2$ neighbours in $C$,
by the pigeonhole principle we can find $x,y$ adjacent on $C$
such that $ax$ and $by$ are edges,
thus extending $C$ to a Hamilton path  from $a$ to $b$.
\end{proof}

\begin{proof}[Proof of \Cref{lem:twocliqueshamilton}] 
Fix disjoint edges $a_1 b_1$ and $a_2 b_2$ with $a_1,a_2 \in A$ and $b_1,b_2 \in B$.
It suffices to find a Hamilton path from $a_1$ to $a_2$ in $G[A]$;
then by symmetry we also have such a path from $b_1$ to $b_2$ in $G[B]$, so $G$ is Hamiltonian.
To do so, we consider the set $L_A:=\{v\in V(G[A])\colon d(v,A)\leq  0.3n\}$ of low degree vertices in $A$,
find two short paths covering $L_A$, then connect them using \Cref{lem:hamiltonconnected}.
 
Counting non-edges in $A$ incident to $L_A$,
we have $\tfrac12 |L_A| (|A|-0.3n) \le 10^{-4} n^2$, so $|L_A| < 2 \cdot 10^{-3} n$.
As  $G[A]$ has minimum degree at least $n/100$, we can greedily choose a `cherry matching' 
where for each $x \in L_A \sm \{a_1,a_2\}$ we choose a path of length $2$ centred at $x$,
with all such paths being vertex-disjoint. If $a_1$ or $a_2$ is in $L_A$ we also choose a disjoint edge
connecting them to a vertex not in $L_A$. 

By definition of $L_A$, any two vertices of $A \sm L_A$
have at least $n/20$ common neighbours in $A \sm L_A$, so we can greedily join all paths chosen so far
to form two vertex disjoint paths $P_1$ from $a_1$ to some $a'_1 \in A \sm L_A$ and
$P_2$ from $a_2$ to some $a'_2 \in A \sm L_A$ that cover $L_A$ 
and have total length at most $n/100$.

The induced graph on $A \sm (V(P_1) \cup V(P_2))\cup \{a_1',a_2'\}$ has minimum degree at least $0.29n$,
so by \Cref{lem:hamiltonconnected} has a Hamiltonian path $P$ from $a'_1$ to $a'_2$.
Then $P_1 P P_2$ is a Hamiltonian path from $a_1$ to $a_2$, as required.
\end{proof}

The required two disjoint edges in $G[A,B]$ in \Cref{lem:twocliqueshamilton} 
will be provided by the following lemma, for which regularity is essential. Given that \Cref{lem:twocliqueshamilton} requires just two edges, a much weaker bound here would suffice, but we include the following nearly optimal bound as it serves as a good warm-up for the upcoming section of the paper.

\begin{lemma} \label{lem:xmatch}
Let $G$ be an $(n+1)$-regular graph on $2n$ vertices
 and $(A,B)$ be a partition of $V(G)$ with $|A|,|B|> \sqrt{n}/100$.
Then $G[A,B]$ has a matching of size $\sqrt {n}/100$.
\end{lemma}

\begin{proof}

Let $M$ be a maximal matching in $G[A,B]$.
Suppose for contradiction that $|M| < \sqrt{n}/100$. 
Assume without loss of generality that $|A|\geq |B|$ and note that $|A|\leq n+\sqrt{n}/100$ as otherwise all vertices in $B$ have degree at least $\sqrt{n}/100$ in $A$, so a matching of size $\sqrt{n}/100$ could be found greedily.

Let $(\wt{A},\wt{B})$ be a balanced cut obtained from $(A,B)$
by moving a set $S$ of at most $\sqrt{n}/100$ vertices.
Let $C$ be a minimal vertex cover in the bipartite graph $G[\wt{A},\wt{B}]$. If $|C|>\sqrt{n}/2$, by K\"onig's theorem this implies the existence of a matching of the same size, and thus we would have a matching of size $\sqrt{n}/2-\sqrt{n}/100$ between $A$ and $B$. Towards a contradiction, suppose that $|C|\le\sqrt{n}/2$. To simplify notation, we write $A, B$ in place of $\tilde{A}, \tilde{B}$ for the remainder of the proof.

Let $A'=A\cap C$ and $B'=B\cap C$. Assume without loss of generality that $e(A',B\setminus B')=t\geq e(B',A\setminus A')$. Since $G$ is $(n+1)$-regular, we have that the complement of $G$ is $(n-1)$-regular, so the number of non-edges $$\overline{e}(A',A\setminus A')\geq |A'|(n-1)-\overline{e}(A',A'\cup B')-\overline{e}(A',B\setminus B')\geq (n-1)|A'|-n/4-(|A'|n-t)> t-n/3,$$ since by assumption $|A'\cup B'|\leq \sqrt{n}/2$. This implies $e(A',A\setminus A')< |A'||A\setminus A'|-t+n/3$, and hence again by regularity and recalling that $e(A\setminus A',B\setminus B')=0$ (as $C$ is a minimal vertex-cover of the bipartite graph) we have
\begin{align*}
    e(A\setminus A', B')&\geq
|A\setminus A'|(n+1)-e(A\setminus A',A\setminus A')-e(A\setminus A',A')\\
&> |A\setminus A'|(n+1)-(|A\setminus A'|-1)|A\setminus A'|-\big(|A'||A\setminus A'|-t+n/3)\\
&\geq 2|A\setminus A'|+t-n/3> t.
\end{align*}
which contradicts the assumption on the maximality of $e(A',B\setminus B')$, and completes the proof.    
\end{proof}


We conclude this subsection by applying  \Cref{lem:twocliqueshamilton} to prove \Cref{lem:near2c}.

\begin{proof}[Proof of \Cref{lem:near2c}]
Let $n^{-1} \ll \eps \ll 1$ and $G$ be an $(n+1)$-regular graph on $2n$ vertices. 
Suppose there is some $A \sub V(G)$ with $n \le|A|\le(1+32\eps)n$
such that $e(A,\ov{A}) \le 24\eps n^{2}$ 
and $G[A]$, $G[\ov{A}]$ both have minimum degree at least $2n/5$.
We note that the number of non-edges in $G[A]$ is at most
$\tbinom{|A|}{2} - (|A|(n+1)-24 \eps n^2)/2 < 40 \eps n^2 < 10^{-4} n^2$ 
for sufficiently small $\eps$, and similarly for $G[\ov{A}]$.
 
Consider $\Gh = G[S]$ where $S$ is a random subset of $V(G)$.
By Chernoff bounds, whp $|S| = n \pm n^{0.6}$, 
$|S \cap A| = |A|/2 \pm n^{0.6}$, $|S \cap \ov{A}| = |\ov{A}|/2 \pm n^{0.6}$,
and $G[S \cap A]$, $G[S \cap \ov{A}]$ both have minimum degree at least $0.1|S|$.
By \Cref{lem:xmatch}, there is a matching of size $0.1n$ in $G[A,\ov{A}]$,
so whp there is a matching size $2$ in $G[S \cap A,S \cap \ov{A}]$  by Chernoff bounds.
Finally, bounding the number of non-edges in $G[S \cap A]$, $G[S \cap \ov{A}]$ by those in $G[A]$, $G[\ov{A}]$,
we can apply  \Cref{lem:twocliqueshamilton} to conclude that $G[S]$ is Hamiltonian.
\end{proof}


\subsection{Almost bipartite}

In this subsection we prove \Cref{lem:nearbip},
which establishes \Cref{thm:main-} for graphs that are almost bipartite.
We use the following sufficient condition for Hamiltonicity in such graphs.
For the statement, we require the following definitions, which will be important throughout the paper.
A \emph{linear forest} is a collection of vertex-disjoint paths.
We say that a  cut $(X,Y)$ of a graph $H$ is \emph{$k$-good} if $|X| \ge |Y|$ 
and $H[X]$ has a linear forest with at least $k+|X|-|Y|$ edges,
or if $|Y| \ge |X|$ and $H[Y]$ has a linear forest with at least $k+|Y|-|X|$ edges;
we call a cut  \emph{good} if it is $0$-good.

\begin{lemma}\label{lem:bipartitehamilton} 
Let $n^{-1} \ll \eps \ll \gG$ and $G$ be a graph on $n$ vertices in which all vertex degrees are $n/2 \pm n^{0.6}$.
Suppose $(A,B)$ is a good cut of $G$ such that $|B| \leq |A|\leq |B|+\eps n$
and $G[A,B]$ has at least $(1/4-\eps)n^2$ edges and minimum degree at least $\gamma n/3$.
Then $G$ has a Hamilton cycle.
\end{lemma}

The proof of \Cref{lem:bipartitehamilton} uses the following bipartite analogue of \Cref{lem:hamiltonconnected}
on Hamilton-connectivity (again we include a short proof for completeness).

\begin{lemma}\label{lem:hamconnectedbipartite}
Let $G=(A,B)$ be a bipartite graph with $|A|=|B|=n/2$ and $\delta(G)\geq n/4+1$. 
Then for any $a\in A$ and $b\in B$ there is a Hamilton path from $a$ to $b$.
\end{lemma}
\begin{proof}
Consider any $a,b\in V(G)$ and obtain $G'$ from $G$ by deleting $a$ and $b$.
Then $G'$ has a Hamiltonian cycle $C$ by \cite[Corollary 1.4]{chvatal1972hamilton}.
Fix an order of $C$ and for each $v \in V(G)$ let $v^+$ denote its successor on $C$.
Then $N_G(b)$ and $\{v^+: v \in N_G(a)\}$ must intersect,
as they are both subsets of $A \sm \{a\}$ of size at least $n/4 + 1$,
so we can extend $C$ to a Hamiltonian path  from $a$ to $b$.
\end{proof}

\begin{proof}[Proof of \Cref{lem:bipartitehamilton}]
Similarly to the proof of \Cref{lem:twocliqueshamilton}, 
our plan will be to find a path $P'$ from some $a' \in A$ to some $b' \in B$
so that $G' = G \sm (V(P') \sm \{a',b'\})$ satisfies the conditions of \Cref{lem:hamconnectedbipartite}.
Then, a Hamilton path from $a'$ to $b'$ will complete $P'$ to the required Hamiltonian cycle.
The path $P'$ will thus need to cover the low degree vertices and balance the part sizes
using the given linear forest in $A$.

We denote the sets of low degree vertices in each part by
 $L_A:=\{v\in A\colon d(v,B)\leq 0.3n\}$ and $L_B:=\{v\in B\colon d(v,A)\leq 0.3n\}$.
To bound $L_A$, we note that $|L_A|(n/2-n^{0.6}-0.3n) \le 2e(A) \le (n/2+n^{0.6})|A|-e(A,B) 
\le (n/2+n^{0.6})(n/2+\eps n)-(1/4-\eps)n^2 < 3\eps n^2$, so $|L_A| < 20\eps n$;
similarly, $|L_B| < 20\eps n$.

Since $|A| \ge |B|$ we can  fix a linear forest $F$
with exactly $|A|-|B| < \eps n$ edges, which exists as $(A,B)$ is a good cut.
Next, as $G[A,B]$ has minimum degree at least $\gG n/3$,
we can greedily choose a cherry matching, where for each $x \in (L_A \cup L_B) \sm V(F)$
we choose a path of length $2$ in $G[A,B]$ centred at $x$, 
with all such paths vertex-disjoint from each other and from $F$.
We can also choose a disjoint edge of $G[A,B]$ for each vertex in $ (L_A \cup L_B) \cap V(F)$
connecting it to a vertex not in $L_A \cup L_B \cup V(F)$.

To complete the construction of the path $P'$, we greedily merge paths,
where in each step we fix two endpoints $x,y$ of two distinct existing paths
and join them by a path of length $2$ or $3$ in $G[A,B]$, vertex-disjoint from the existing paths.
To see that this is possible, note that if $x,y$ are in the same part then they have 
at least $0.3n + 0.3n - (n/2 + \eps n) > n/20$ common neighbours in the other part,
so we can choose the required path of length $2$, or if $x \in A$, $y \in B$ then we can
connect $x$ to some new $x' \in B$ then connect $x',y$ by a disjoint path of length $2$.
This process ends with one path $P'$ of length $<10^3 \eps n$; 
by possibly adding one more edge we can assume it has ends $a' \in A$ and $b' \in B$.

Finally, $G' = G \sm (V(P') \sm \{a',b'\})$ has balanced parts $(A',B')$ by choice of $F$, 
and $G'[A',B']$ has minimum degree at least $0.3n - 10^3 \eps n > |A'|/2 + 1$,
so it satisfies the conditions of \Cref{lem:hamconnectedbipartite},
which gives a Hamiltonian path in  $G'[A',B']$  from $a'$ to $b'$ 
that completes $P'$ to the required Hamiltonian cycle.
\end{proof}

The required linear forest in \Cref{lem:bipartitehamilton}
will be provided by the following lemma, which uses regularity,
and plays an essential role here and later in the paper.

We recall that a balanced cut $(A,B)$
is a partition of $V(G)$ with $|A|=|B|=n$. 

\begin{lemma}\label{lem:balancedcutbigmatching}
Let $G$ be an $(n+1)$-regular graph on $2n$ vertices and $(A,B)$ be a balanced cut.
Suppose $A',B'$ are vertex covers of $G[A],G[B]$. Then $(|A'|+1)(|B'|+1)\geq n+1$.    
\end{lemma}
\begin{proof}
Write $V:=V(G)$. Suppose without loss of generality that $e(A\sm A',A')\geq e(B\sm B',B')$.
Note that 
 \begin{align*}
    e(A',B\sm B')\leq e(A',B)=e(A',V)-e(A',A')-e(A',A\sm A')\leq (n+1)|A'|-e(A',A\sm A').  
 \end{align*}  
As $B \sm B'$ is an independent set, we deduce
\begin{align*}
    e(B\sm B',B')&=e(B\sm B', V)-e(B\sm B', A \sm A')-e(B\sm B', A')-e(B\sm B', B\sm B')\\
    &\geq (n+1)(n-|B'|)-(n-|A'|)(n-|B'|)-((n+1)|A'|-e(A', A\sm A'))-0\\
    &= (n+1)-(|A'|+1)(|B'|+1)+e(A',A\sm A').
\end{align*}
As $e(A\sm A',A')\geq e(B\sm B',B')$, we conclude that $(n+1)-(|A'|+1)(|B'|+1)\leq 0$.
\end{proof}

The following simple remarks will be used throughout the remainder of the paper.
\begin{remark} \label{rem:simple}
Let $M$ be a maximal matching in a graph $H$. 
Then $V(M)$ is a vertex cover in $H$. In particular, $|M|\geq |C|/2$ where $C$ is a minimum vertex cover.
Also, suppose $H$ has maximum degree $\DD$
and $C$ is a vertex cover in $H$. 
Then  $e(H) \le \DD |C|$.
\end{remark}

As the size of a random subset of a set $A$ has the binomial $B(|A|,1/2)$ distribution,
we will frequently need the following normal approximation, which follows from the Central Limit Theorem.

\begin{lemma} \label{lem:CLT}
If $X \sim B(n,1/2)$ then $\mb{P}(a\sqrt{n}/2 \le X-n/2 \le b\sqrt{n}/2) = \int_a^b \tfrac{1}{\sqrt{2\pi}} e^{-t^2/2} \ dt + o(1)$ for any $a<b$.
\end{lemma}

Finally, we need the following well-known observation on the difference of independent binomials.

\begin{lemma} \label{lem:bindiff}
If $X \sim B(n,1/2)$,  $Y \sim B(m,1/2)$ are independent then $X+m-Y \sim B(n+m,1/2)$.
\end{lemma}

\begin{proof}
Let $I_1,\dots,I_{n+m}$ be iid Bernoulli$(1/2)$ variables.
We can write $X =  \sum_{i=1}^n I_i$ and $Y =  \sum_{i=n+1}^{n+m} (1-I_i)$, 
as each $1-I_i$ is also Bernoulli$(1/2)$.
Then $X+m-Y =  \sum_{i=1}^{n+m} I_i\sim Bin(n+m,1/2)$.
\end{proof}

We conclude this subsection by applying \Cref{lem:bipartitehamilton} to prove \Cref{lem:nearbip},
thus completing the proof of \Cref{thm:main-}, and so \Cref{conj:EF}. 

\begin{proof}[Proof of \Cref{lem:nearbip}]
Let $n^{-1} \ll \eps \ll \gG \ll \dD  \ll 1$ and $G$ be an $(n+1)$-regular graph on $2n$ vertices. 
Suppose there is some $A \sub V(G)$ with $n \le|A|\le(1+32\eps)n$
such that $G[A,\ov{A}]$ has at least $(1-56\eps)n^{2}$ edges 
and minimum degree at least $\gG n$, where if  $|A|>n$ then
 $G[A]$ has maximum degree at most $\gG n$.

Consider $\Gh = G[S]$ where $S$ is a random subset of $V(G)$.
By Chernoff bounds, whp $|S| = n \pm n^{0.6}$,
all vertex degrees are $|S|/2 \pm |S|^{0.6}$,
$|S \cap A| = |A|/2 \pm n^{0.6}$, $|S \cap \ov{A}| = |\ov{A}|/2 \pm n^{0.6}$
and $G' := G[S \cap A,S \cap \ov{A}]$ has minimum degree at least $\gamma n/3$.

Also, $G'$ has at most as many non-edges as $G[A,\ov{A}]$,
so $e(G') \ge |A\cap S||\ov{A}\cap S|-56\eps n^{2} \ge (1/4 - 150\eps)n^2$.

To complete the proof via  \Cref{lem:bipartitehamilton} (with $\eps$ replaced by $150\eps$), 
it remains to show that the event $E$ that $(S \cap A, S \cap \ov{A})$ is a good cut of $G[S]$ 
has $\mb{P}(E) \ge 0.01$.

To do so, we consider two cases according to the size of $k := |A|-n$.

\medskip

\nim{Case 1:} $k > \dD\sqrt{n}$.

\medskip

Here $|A|>n$, so $G[A]$ has maximum degree at most $\gG n$, by assumption.
Also $G[A]$ has minimum degree at least $\dD(G)-|B| \ge k+1$,
so by \Cref{rem:simple} every vertex cover has size at least $e(G[A])/(\gG n)\geq nk/(2\gG n)=k/2\gamma$, and hence $G[A]$ has a matching of size at least $k/4\gamma$.
As $k > \dD\sqrt{n}$ and $\gG \ll \dD$, by Chernoff bounds
$\mb{P}[|M[S]| > k/20\gG] > 1 - \dD$, say.
Also, by \Cref{lem:bindiff} and Chernoff we have
$\mb{P}[0 \le |S \cap A |-|S \cap \ov{A}| \le k/20\gG] 
= \mb{P}[n-k \le B(2n,1/2) \le n-k + k/20\gG]  > 1/2 - \dD$, 
say, so $\mb{P}(E) \ge 1/2 - 2\dD$.

\medskip

\nim{Case 2:} $k \le \dD\sqrt{n}$.

\medskip

We consider any balanced cut $(A^*,B^*)$ obtained from $(A,\ov{A})$
by moving $k$ vertices from $A$ to $\ov{A}$
and minimum vertex covers $A',B'$ of $G[A^*],G[B^*]$.
By \Cref{lem:balancedcutbigmatching}, we have $\max \{|A'|,|B'|\} \ge \sqrt{n}-1$.

Suppose first that $|B'|  \ge \sqrt{n}-1$. 
Then $G[B^*]$ has a matching $M$ of size $(\sqrt{n}-1)/2$ by \Cref{rem:simple}.
By Chernoff we have $\mb{P}[|M[S]| > \sqrt{n}/9] > 0.99$, say.
At most $k \le \dD\sqrt{n}$ edges in $M(S)$ contain a moved vertex,
so this event gives a matching of size $0.1\sqrt{n}$ in $S \cap \ov{A}$.
Also, by \Cref{lem:bindiff} and \Cref{lem:CLT},
\[ \mb{P}[0 \le |S \cap \ov{A}|-|S \cap A | \le 0.1\sqrt{n}] 
= \mb{P}[n+k \le B(2n,1/2) \le n+k + 0.1\sqrt{n}],\]
which by  \Cref{lem:CLT} differs by at most $\dD$
from $\int := \int_0^{0.1/\sqrt{2}} \tfrac{1}{\sqrt{2\pi}} e^{-t^2/2} \ dt > 0.02$.
Thus  $\mb{P}(E) \ge 0.01$.

Similarly if $|A'|  \ge \sqrt{n}-1$,
we have a matching of size $0.1\sqrt{n}$ in $S \cap A$ with probability $>0.99$
and \[ \mb{P}[0 \le |S \cap A |-|S \cap \ov{A}| \le 0.1\sqrt{n}] 
= \mb{P}[n-k \le B(2n,1/2) \le n-k + 0.1\sqrt{n}] \] 
again differs by at most $\dD$ from $\int > 0.02$.
In all cases,  $\mb{P}(E) \ge 0.01$, as required.
\end{proof}


\section{Asymptotic result} \label{sec:asymptotic}

In this section we strengthen our solution of \Cref{conj:EF}
to the following asymptotically tight result.

\begin{theorem}\label{thm:main} 
Any  $(n+1)$-regular graph $G$ on $2n$ vertices 
has $\mb{P}(\Gh \text{ is Hamiltonian}) > 1/2 - o(1)$.
\end{theorem}

In the first subsection we give the proof of \Cref{thm:main},
assuming a key lemma (\Cref{lem:half+}), which will also
be used in the proof of our exact result \Cref{thm:main+} in the next section.
In the second subsection we prove the key lemma, assuming a technical lemma 
that we prove by elementary calculus in the final subsection.

\subsection{The key lemma} \label{sub:key}

Here we prove \Cref{thm:main} assuming the following key lemma.
For the statement, recall that a  cut $(X,Y)$ of a graph $H$ is \emph{$k$-good} 
if $|X| \ge |Y|$ and $H[X]$ has a linear forest with at least $k+|X|-|Y|$ edges,
or if $|Y| \ge |X|$ and $H[Y]$ has a linear forest with at least $k+|Y|-|X|$ edges.
 
 \begin{lemma}\label{lem:half+}  
Let $n^{-1} \ll  \dD \ll \lL \ll \eta \ll 1$.
Let $G$ be an $(n+1)$-regular graph on $2n$ vertices. 
Let $(A,B)$ be a balanced cut of $G$.
Let $A',B'$ be minimum vertex-covers of $G[A],G[B]$
of sizes $\aA \sqrt{n}$, $\bB \sqrt{n}$.
Suppose that $\min\{\aA,\bB\} > \eta$.
Let $G[S]$ be a random induced subgraph of $G$.
Then $(A \cap S,B \cap S)$ is a $\dD\sqrt{n}$-good cut of $G[S]$
with probability at least $1/2 + \lL$.
\end{lemma}

\begin{proof}[Proof of  \Cref{thm:main}]
Let $n^{-1} \ll \eps \ll \gG \ll \dD \ll \eta \ll \tT \ll 1$ 
and $G$ be an $(n+1)$-regular graph on $2n$ vertices.
Let $\Gh = G[S]$ be a random induced subgraph of $G$.
We will show $\mb{P}(G[S] \text{ is Hamiltonian}) > 1/2 - \tT$.
As in the proof of \Cref{thm:main-}, 
we consider three cases according to \Cref{lem:casesdirac},
where in the first two cases
$\mb{P}(G[S] \text{ is Hamiltonian}) = 1 - o(1)$
by  \Cref{lem:bidense} and \Cref{lem:near2c}.
Thus it suffices to consider the third near-bipartite case:
there is some $A \sub V(G)$ with $n \le|A|\le(1+32\eps)n$
such that $G[A,\ov{A}]$ has at least $(1-56\eps)n^{2}$
edges and minimum degree at least $\gamma n$, where if  $|A|>n$ then
 $G[A]$ has maximum degree at most $\gamma n$.
 
 As in the proof of \Cref{lem:nearbip}, whp $G[S]$ satisfies all conditions
 of \Cref{lem:bipartitehamilton}, with the possible exception of 
 the event $E$ that $(S \cap A, S \cap \ov{A})$ is a good cut of $G[S]$.
 To complete the proof, it suffices to show $\mb{P}[E] > 1/2 - \tT/2$.
 Again, we consider two cases according to $k := |A|-n$.
 If $k > \dD\sqrt{n}$ then $\mb{P}(E) \ge 1/2 - 2\dD$ 
 as in Case 1 of the proof of \Cref{lem:nearbip}.
 Thus we can assume $k \le \dD\sqrt{n}$.
 As in Case 2 of the proof of \Cref{lem:nearbip},
 we consider any balanced cut $(A^*,B^*)$ obtained from $(A,\ov{A})$
by moving $k$ vertices from $A$ to $\ov{A}$
and minimum vertex covers $A',B'$ of $G[A^*],G[B^*]$ of sizes  $\aA \sqrt{n}$, $\bB \sqrt{n}$.
By \Cref{lem:balancedcutbigmatching}, we have $\max \{|A'|,|B'|\} \ge \sqrt{n}-1$.

Let $E'$ be the event that $(S \cap A^*, S \cap B^*)$ is a $k$-good cut of $G[S]$.
Then $E'$ implies $E$, so it suffices to show $\mb{P}[E'] > 1/2 - \tT/2$.
This holds by \Cref{lem:half+}  if $\min\{\aA,\bB\} \ge \eta$, 
so we can assume otherwise, say $\aA<\eta$ and $\bB > \tfrac12 \eta^{-1}$.
Then $G[B^*]$ contains a matching of size $\tfrac14 \eta^{-1} \sqrt{n}$ by \Cref{rem:simple}.
By Chernoff bounds, 
whp $G[S \cap B^*]$ contains a matching of size $\tfrac{1}{20} \eta^{-1} \sqrt{n}$,
and as $\eta \ll \tT$, with probability at least $1/2-\tT/2$ 
we have $0 \le |B \cap S|-|A \cap S| \le \tfrac{1}{20} \eta^{-1} \sqrt{n}$.
The result follows.
 \end{proof}

\subsection{Proof of the key lemma} \label{sub:pfkey}

In this subsection we prove \Cref{lem:half+}, 
assuming \Cref{lem:calculus} below,
which will be proved in the following subsection.
To achieve the asymptotically optimal probability 
of finding a good cut in $\Gh$ we will require 
a much tighter argument for finding linear forests,
based on the following lemma.

\begin{lemma}\label{lem:edges and max degree}
Let $G$ be an $(n+1)$-regular graph on $2n$ vertices. 
Let $(A,B)$ be a balanced cut of $G$.
Let $A',B'$ be minimal vertex-covers of $G[A],G[B]$
of sizes $\aA \sqrt{n}$, $\bB \sqrt{n}$.
Suppose $\alpha, \beta =\Theta(1)$
and $\bar{e}(A', B\sm B')\geq \bar{e}(B', A\sm A')$. 
Then there exist (bipartite) subgraphs $G_A$ of $G[A', A\sm A']$ and  $G_B$ of $G[B', B\sm B']$ with
\[ e(G_B) \ge n-O(\sqrt{n}), \quad \DD(G_B) \le \aA\sqrt{n}+1,
\quad e(G_A) \ge (2-\aA\bB)n-O(\sqrt{n}), \quad  \DD(G_A) \le \bB\sqrt{n}+1.\]
\end{lemma}

\begin{proof}
Recall for any $X,Y \sub V(G)$ and $v \in V(G)$
that $\bar{e}(X,Y)$ counts non-edges in $X \times Y$
and $\bar{d}(v,X)$ counts non-neighbours of $v$ in $X$.
Note that as $G$ is $(n+1)$-regular and $|A|=|B|=n$, 
for $v \in B$ we have $d(v,B)=t$ if and only if $\bar{d}(v,A)=t-1$,
and similarly interchanging $A$ and $B$;
we will use this observation repeatedly below.

Write $cn:=\bar{e}(A', B\sm B')$ and $c'n=\bar{e}(B', A\sm A')$. 
Then $c'\leq c$ by assumption. 

We also write $sn:=\bar{e}(A\sm A', B\sm B')$. 
As $B \sm B'$ is independent we have
\begin{equation}\label{3} e(B\sm B', B')= \sum_{v\in B\sm B'} (\bar{d}(v, A)+1) 
 \geq  (c+s)n + |B\sm B'|\geq (1+c+s)n- \bB\sqrt{n}.\end{equation} 

Let $B_1 = \{v \in B': d(v, B\sm B')\geq \aA\sqrt{n} \}$
and $B_2 =  \{v \in B \sm B': d(v, B')\geq \aA\sqrt{n} \}$.

For $v \in B_1$ we have $\bar{d}(v,A \sm A') \ge \bar{d}(v,A)-|A'| \ge  d(v, B\sm B')-1-\aA\sqrt{n}$.

Similarly, for $v \in B_2$ we have $\bar{d}(v,A \sm A') \ge d(v, B')-1-\aA\sqrt{n}$, so
\[  \sum_{v \in B_1} (d(v, B\sm B')-1-\aA\sqrt{n}) \le c'n\le cn, \qquad 
    \sum_{v \in B_2} (d(v, B')-1-\aA\sqrt{n}) \le sn. \]

Thus we can delete at most $sn+cn$ edges from $G[B', B \sm B']$
to obtain $G_B$ with maximum degree at most  $\aA\sqrt{n}+1$ 
and $e(G_B) \ge n-\bB\sqrt{n}$ by \eqref{3}.

From \eqref{3} we also obtain 
$\bar{e}(B', A) = \sum_{v \in B'} \bar{d}(v,A) =   \sum_{v \in B'} (d(v,B)-1) \ge (1+c+s)n- 2\bB\sqrt{n}$, so
\[ c'n = \bar{e}(B', A\sm A') \ge \bar{e}(B', A) - |A'||B'| \ge (1+c+s-\aA\bB)n- 2\bB\sqrt{n}. \]
Similarly to \eqref{3}, as $A \sm A'$ is independent we deduce
\[ e(A\sm A', A')\geq \sum_{v\in A\sm A'} (\bar{d}(v, B)+1) 
\ge \bar{e}(B', A\sm A') + |A \sm A'|
\ge (2+c+s-\aA\bB)n - (\aA+2\bB)\sqrt{n}.\]
Now by the same argument as for $G_B$ 
we can delete at most $sn+cn$ edges from $G[A', A \sm A']$
to obtain $G_A$ with maximum degree at most  $\bB\sqrt{n}+1$ 
and $e(G_A) \ge  (2-\aA\bB)n - O(\sqrt{n})$.
 \end{proof}

We also need to sharpen our argument for passing from minimum vertex covers in $G$ to matchings in $\Gh$.
Previously we only used minimality and lost a factor of $8$, using the naive estimate that one may lose a factor of $2$ 
in finding a matching in $G$ and then each edge survives with probability $1/4$ in $\Gh$.
In the following lemma we consider a \emph{minimum vertex cover},
meaning that it is not only minimal but also of minimum possible size,
and improves the factor $8$ to $4$ under mild assumptions on the size.

\begin{lemma}\label{lem: cover over 4} 
Let $G$ be a graph on vertices with minimum vertex cover $C$, where $\log^2 n \le |C| \le n/4$.
Then whp $\Gh$ has a matching of size $(1-o(1))|C|/4$.
\end{lemma}

\begin{proof}
Write $\Gh = G[S]$, where $S$ is a uniformly random subset of $V(G)$.
We condition on $C' := S \cap C$ and fix a maximal matching $M$ in $G[C']$.
By Chernoff whp $|C'| \ge (1-o(1))|C|/2$.
We can assume $|M| < (1-o(1))|C'|/2$, otherwise we are done.
Then $C'' := C' \sm V(M)$ is an independent set, by maximality of $M$.

Next we note that the bipartite graph $G'=G[C'', V(G) \sm C]$ satisfies the Hall condition 
that $|N_{G'}(X)|\geq |X|$ for every $X\subseteq C''$. Indeed, if this failed for some $X$
then $(C \sm X) \cup N_{G'}(X)$ would be a vertex cover of $G$
that contradicts $C$ having minimum possible size.
Thus there is a matching $M'$ in $G'$ which covers $C''$.

Now we reveal the rest of $S \sm C$. Each edge of $M'$ survives in $M'[S]$
with probability $1/2$, so by Chernoff whp $(1-o(1))|C''|/2$ edges survive.
Combining $M$ and $M'[S]$ gives the required matching as
\[ |M \cup M'[S]| = |M| + (1-o(1))(|C'|-|M|)/2 \ge (1-o(1))|C'|/2 \ge (1-o(1))|C|/4. \qedhere \]
\end{proof}

We will obtain linear forests via a well-known result of Alon \cite{alon1988linear} (see also \cite{lang2023improved})
giving an approximate form of the linear arboricity conjecture.

\begin{theorem}[Asymptotic linear arboricity]\label{lineararboricity}
    Let $1/\Delta\ll \eps$ and $G$ be a graph of maximum degree at most $\Delta$. 
    Then $G$ can be edge-decomposed into at most $(1+\eps)\Delta/2$ linear forests. 
\end{theorem}

So far we have only relied on the Chernoff concentration inequality,
but for the asymptotically correct count on edges in random induced subgraphs
we also need the following consequence of Azuma's inequality \cite{azuma1967}.

\begin{lemma} \label{lem:induce}
Let $G$ be a graph on $n$ vertices with $\DD(G) = o(e(G))$.
Then whp $e(\Gh) = (1+o(1))e(G)/4$.
\end{lemma}

\begin{proof}
Let $X=e(S)$, where $S$ is a uniformly random subset of $V(G)$. Then $\mb{E}X=e(G)/4$, and  
changing whether some vertex $v$ is in $S$ can affect $e(\Gh)$ by at most $d_G(v)$.
By Azuma's inequality $\mb{P}(|X-\mb{E}X|>t) \le e^{-t^2/2Q}$, where 
$Q = \sum_{v \in V(G)} d(v)^2 \le \DD(G)  \sum_{v \in V(G)} d(v) = 2\DD(G)e(G) = o(e(G)^2)$.
The lemma follows.
\end{proof}

We conclude this subsection with the proof of the key lemma,
assuming the following calculus lemma to be proved in the next subsection.

\begin{lemma}\label{lem:calculus}
Let $n^{-1} \ll  \dD \ll \lL \ll \eta \ll 1$
and $\alpha,\beta>\eta$. 
Define $m_1 = \max\{\alpha/4, 2/\bB-\alpha\}$
and $m_2 =\max\{\beta/4, 1/\alpha\}$.
Let $X,Y \sim B(n,1/2)$ be independent binomials.
Consider the event $E$ that 
$-(m_1-\dD)\sqrt{n} \le X-Y \le (m_2-\dD)\sqrt{n}$. 
Then $\mb{P}[E] \ge 1/2 + \lL$.
\end{lemma}

\begin{proof}[Proof of \Cref{lem:half+}]
Let $n^{-1} \ll  \dD \ll \lL \ll \eta \ll 1$.
Let $G$ be an $(n+1)$-regular graph on $2n$ vertices
and $(A,B)$ be a balanced cut of $G$.
Let $A',B'$ be minimum vertex-covers of $G[A],G[B]$
of sizes $\aA \sqrt{n}$, $\bB \sqrt{n}$,
where $\min\{\aA,\bB\} > \eta$.
Let $\Gh=G[S]$ be a random induced subgraph of $G$.
Let $E'$ be the event that $(A \cap S,B \cap S)$ is a $\dD\sqrt{n}$-good cut of $\Gh$.
We will show $\mb{P}[E'] \ge \tfrac12 + \lL$.

We assume without loss of generality that
$\bar{e}(A', B\sm B')\geq \bar{e}(B', A\sm A')$. 

As $\min\{\aA,\bB\} > \eta$, we can apply \Cref{lem:edges and max degree},
obtaining subgraphs $G_B$ of $G[B', B\sm B']$ and $G_A$ of $G[A', A\sm A']$ with
$e(G_B) \ge n-O(\sqrt{n})$, $\DD(G_B) \le \aA\sqrt{n}+1$,
$e(G_A) \ge (2-\aA\bB)n-O(\sqrt{n})$ and $\DD(G_A) \le \bB\sqrt{n}+1$.

By \Cref{lem:induce} and Chernoff whp $G'_A := G_A[A \cap S]$ 
has $e(G'_A) \ge (2-\alpha\beta-o(1))n/4$ and $\DD(G_A) \le (1+o(1))\bB\sqrt{n}/2$,
and $G'_B := G_B[B \cap S]$ 
has $e(G'_B) \ge (1-o(1))n/4$ and $\DD(G_B) \le (1+o(1))\aA\sqrt{n}/2$.

By \Cref{lineararboricity}, we decompose $G'_A$ into $(1+o(1))\bB \sqrt{n}/4$ linear forests,
so by averaging we find a linear forest with $(2/\bB - \aA - o(1)) \sqrt{n}$ edges. 
Similarly, in $G'_B$ we find a linear forest with $(1/\aA - o(1)) \sqrt{n}$ edges. 
 
Also, as $\min\{\aA,\bB\} > \eta$, by \Cref{lem: cover over 4} whp we have matchings of size
$(\aA/4-\dD)\sqrt{n}$ in $G[A \cap S]$ and $(\bB/4-\dD)\sqrt{n}$ in $G[B \cap S]$.
Thus whp $G[A \cap S]$ contains a linear forest with $(m_1-\dD) \sqrt{n}$ edges
and $G[B \cap S]$ contains a linear forest with $(m_2-\dD) \sqrt{n}$ edges,
where $m_1 = \max\{\alpha/4, 2/\bB-\alpha\}$
and $m_2 =\max\{\beta/4, 1/\alpha\}$, 

Finally, applying \Cref{lem:calculus} to $|A \cap S|, |B \cap S| \sim B(n,1/2)$,
replacing $(\dD,\lL)$ by $(2\dD,2\lL)$, with probability at least $1/2+2\lL$
we have $-(m_1-2\dD)\sqrt{n} \le |B \cap S|-|A \cap S|  \le (m_2-2\dD)\sqrt{n}$, so $E'$ holds.
\end{proof}

\subsection{The technical lemma}

We conclude this section by proving \Cref{lem:calculus},
thus completing the proof of the key lemma,
and so of the asymptotic result \Cref{thm:main}.

\begin{proof}[Proof of  \Cref{lem:calculus}]
Let $n^{-1} \ll  \dD \ll \lL \ll \eta \ll 1$ and $\alpha,\beta>\eta$.
Let $X,Y \sim B(n,1/2)$ be independent binomials
and $E$ be the event that $-(m_1-\dD)\sqrt{n} \le X-Y \le (m_2-\dD)\sqrt{n}$,
where $m_1 = \max\{\alpha/4, 2/\bB-\alpha\}$
and $m_2 =\max\{\beta/4, 1/\alpha\}$.
We will show $\mb{P}[E] \ge 1/2 + \lL$.
As $X+n-Y \sim B(2n,1/2)$ by \Cref{lem:bindiff},
 by \Cref{lem:CLT} for any $a<b$ we have
$\mb{P}(a\sqrt{n} \le X-Y \le b\sqrt{n}) = I[a,b] + o(1)$,
where $I[a,b] := \int_{a\sqrt{2}}^{b\sqrt{2}} \tfrac{1}{\sqrt{2\pi}} e^{-t^2/2} \ dt$. 
It suffices to show $I[-m_1,m_2] \ge 1/2 + 2\lL$,
as then by continuity $\mb{P}[E] \ge 1/2 + \lL$ for $\dD \ll \lL$. 

Denote $f(\aA) :=  I[-\tfrac{\alpha}{4}, \tfrac{1}{\alpha}]$.
As $\alpha/4 \le 2/\bB-\alpha$ if and only if $\bB \le \tfrac{8}{5\aA}$,
while $\beta/4 \le 1/\alpha$ if and only if $\bB \le \tfrac{4}{\aA}$, we have
\[
I[-m_1,m_2] =
\begin{cases} 
I\left[\alpha-\frac{2}{\beta}, \frac{1}{\alpha}\right] \ge f(\aA) & \text{if } 0<\beta \leq \frac{8}{5\alpha}, \\[10pt]
I\left[-\frac{\alpha }{4}, \frac{1}{\alpha}\right] = f(\aA) & \text{if } \frac{8}{5\alpha} < \beta < \frac{4}{\alpha}, \\[10pt]
I\left[-\frac{\alpha}{4}, \frac{\beta}{4}\right] \ge f(\aA) & \text{if } \beta \geq \frac{4}{\alpha}.
\end{cases}
\]
Thus it remains to show $f(\aA) \ge 1/2 + 2\lL$. 
\Cref{fig:example} shows a computational illustration of this inequality.

\begin{figure}[h!] 
    \centering
    \includegraphics[width=0.5\textwidth]{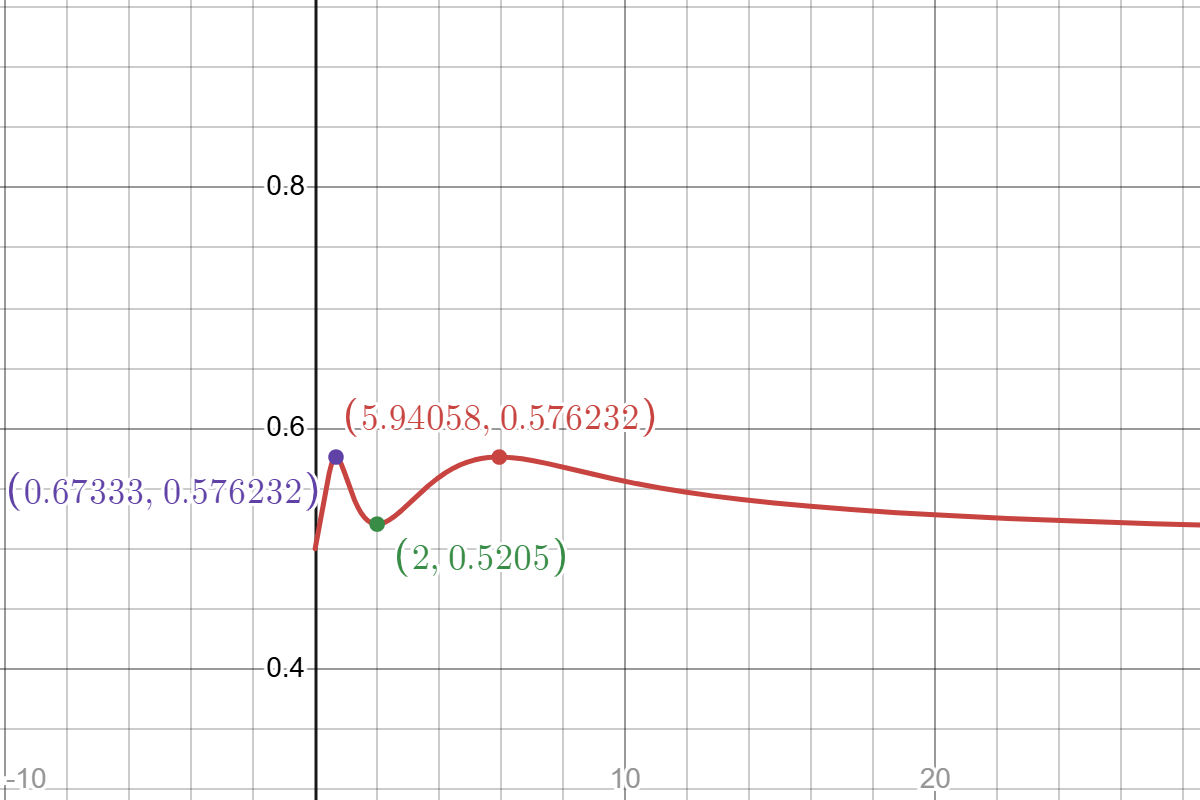} 
    \caption{An illustration of the function $f(\aA)$ and its local extrema.} 
    \label{fig:example} 
\end{figure}

The reader who is satisfied with this computational evidence can skip the remainder of the proof,
in which we verify the inequality using calculus. As $f$ is continuous and $\lL \ll \eta$, 
it suffices to show that $f(\aA)>1/2$ for all $\aA>0$.
As $\aA \to 0$ or $\aA \to \infty$ we have $f(\aA) \to I[0,\infty]=1/2$.
We will show that $f$ is at first increasing, until it reaches a local maximum,
then it is decreasing until it reaches a local minimum at $\aA=2$,
then it is increasing again until it reaches another local maximum,
after which it is again decreasing. One can verify $f(2) = I[-1/2,1/2] >1/2$
via tables of the normal distribution, so this will complete the proof.

We set  $x:=\alpha^2$ and calculate the derivative of $f$ via the Leibniz rule:
\[
f'(\aA)= \frac{1}{\sqrt{2\pi}} e^{-(\sqrt{2}\alpha/4)^2/2}\cdot \frac{\sqrt{2}}{4}
-\frac{1}{\sqrt{2\pi}} e^{-(\sqrt{2}/\alpha)^2/2}\cdot \frac{\sqrt{2}}{\alpha^2}
= \frac{1}{\sqrt{2\pi}} \Big( \tfrac14 e^{-x/16} - \tfrac1x e^{-1/x} \Big).
\]
Considering $x=4$, we observe that $f'(2)=0$, as illustrated by the green dot in \Cref{fig:example}. 

To analyse the sign of $f'(\aA)$, it will be more convenient to consider
\[ g(x):=-\tfrac{x}{16} + \tfrac{1}{x}  + \ln \tfrac{x}{4},
\qquad g'(x) = -\tfrac{1}{16} - \tfrac{1}{x^2}  + \tfrac{1}{x},
\]
noting that $g(x)$ has the same sign as $f'(\aA)$ and is zero iff $f'(\aA)$ is zero.

We also note that $g'(x) > 0$ iff  $-\frac{x^2}{16}+x-1>0$, 
where the quadratic roots are $8 \pm 4\sqrt{3}$,
so $g(x)$ is decreasing in $R_1 := (0,8-4\sqrt{3})$, 
increasing in $R_2 := (8-4\sqrt{3},8+4\sqrt{3})$
and decreasing in $R_3 := (8+4\sqrt{3},\infty)$.

Thus $g(x)$ has at most one root in each $R_i$,
one of which we have seen is $g(4)=0$ with $r_2 = 4 \in R_2$.

As $g'(4)>0$, if $g(x)$ had no root in $R_1$ then by the Intermediate Value Theorem it would be negative in $R_1$.
However, $g(x) \to \infty$ as $x \to 0$, so $g(x)$ must have a root $r_1 \in R_1$,
so $g(x)>0$ in $(0,r_1)$ and $g(x)<0$ in $(r_1,r_2)$.

Similarly, if $g$ had no root in $R_3$ it would be positive in $R_3$,
but $g(x) \to -\infty$ as $x \to \infty$, so $g$ must have a root $r_3 \in R_3$.
Then $g(x)>0$ in $(r_2,r_3)$ and $g(x)<0$ for $x>r_3$. 

This verifies the properties of $f(\aA)$ that we showed above suffice to complete the proof.
\end{proof}

\section{Exact result} \label{sec:exact}

In this section we prove our exact result, \Cref{thm:main+} in the following probabilistic form.
For any graph $G$, let $p(G)$ be the probability that the random induced subgraph $\Gh$ is Hamiltonian.
We need to show for any $(n+1)$-regular graph $G$ on $2n$ vertices with $n$ sufficiently large
that $p(G) \ge p_n := \min_{G' \in \mc{G}_n} p(G')$.

The proof requires a more careful implementation of the strategy for the asymptotic result in the previous section,
and also more accurate estimates for $p(G)$. Our first lemma estimates the second order term in $p_n$.

\begin{lemma} \label{lem:pn}
We have $p_n = \tfrac12 + \tfrac{3/2}{\sqrt{n\pi}} + O(n^{-3/2})$.
\end{lemma}
\begin{proof} 
Fix $G \in \mc{G}_n$, with parts $(A,B)$, where $|A|=n+1$ and $|B|=n-1$.
Recall that $G$ is obtained from the complete bipartite graph $A \times B$
by adding some $2$-factor inside $A$.
Then $\Gh=G[S]$ is Hamiltonian if and only if $|A \cap S| \ge |B \cap S|$ 
and $G[A \cap S]$ contains a linear forest of size $|A \cap S| - |B \cap S|$.
By Chernoff bounds, with probability $1-e^{-\Theta(n)}$ we have $||A \cap S|-|B \cap S|| < n/100$ 
and $G[A \cap S]$ contains a linear forest of size $n/100$.

Thus $p_n =  \mb{P}(|A \cap S| \ge |B \cap S|) \pm e^{-\Theta(n)}$.
By \Cref{lem:bindiff} we have
$ \mb{P}(|A \cap S| \ge |B \cap S|) = \mb{P}(B(2n,1/2) \ge n-1)$.
As $\mb{P}(B(2n,1/2) \ge n+1) = \mb{P}(B(2n,1/2) \le n-1) = (1-\mb{P}(B(2n,1/2)=n))/2$,
we deduce the required estimate
$\mb{P}(B(2n,1/2) \ge n-1) = \tfrac12 + \tfrac{3/2}{\sqrt{n\pi}} \pm n^{-3/2}$,
using $4^{-n} \tbinom{2n}{n} = (n\pi)^{-1/2} \pm n^{-3/2}$ by Stirling's formula.\end{proof} 
The following lemma gives an estimate for binomial probabilities
that improves on the Chernoff bound for small or moderate deviations.
\begin{lemma} \label{lem:compare}
Let $n\in \mathbb N$ be large, and let $k,k_1,k_2\in \mathbb Z$ be such that $k\in[k_1,k_2]\subseteq[-n,n]$.
Write $f_n(t) := \mb{P}(B(2n,1/2) \ge n+t)$. Then
\[ \mb{P}(B(n+k,1/2) - B(n-k,1/2)) \in [k_1,k_2] ) 
= 1 - f_n(k-k_1+1) - f_n(k_2-k+1),\]
where $f_n(t) \le e^{-t^2/(3n+t)}$,
and for $|t|\leq \sqrt{n}/100$ we have $f_n(t) = 1/2 - \tfrac{t-1/2}{\sqrt{n\pi}} \pm (2t^3+1)/n^{3/2}$.\end{lemma}
\begin{proof}
We have $\mb{P}(B(n+k,1/2) - B(n-k,1/2)) \in [k_1,k_2] ) 
 =  \mb{P}(B(2n,1/2) \in [n-k+k_1,n-k+k_2] )$ by  \Cref{lem:bindiff},
so the stated equality holds.
The first estimate on $f_n(t)$ follows from the Chernoff bound.
For the second estimate, we may assume $t\geq 0$ by symmetry. 
Note that $$\mb{P}[B(2n,1/2)=n+t]=\frac{\binom{2n}{n+t}}{\binom{2n}{n}}\mb P[B(2n,1/2)=n]= (1\pm 2t^2/n)(\tfrac{1}{\sqrt{n\pi}} \pm n^{-3/2}).$$
\pagebreak
Hence, using $\mb{P}(B(2n,1/2) \ge n-1) = \tfrac12 + \tfrac{3/2}{\sqrt{n\pi}} \pm n^{-3/2}$, we conclude
\[
f_n(t)=\mb P[B(2n,1/2)\geq n-1]-\sum_{i=n-1}^{n+t-1} \mb P[B(2n,1/2)=i]=1/2-\frac{t-1/2}{\sqrt{n\pi }}\pm (2t^3+1)/n^{3/2}.
\qedhere
\]
\end{proof}

We conclude by proving our main theorem.

\begin{proof}[Proof of \Cref{thm:main+}]
Let $n^{-1} \ll \eps \ll \gG \ll \dD \ll \eta \ll \tT \ll 1$
and $G$ be an $(n+1)$-regular graph on $2n$ vertices. 
Let $\Gh = G[S]$ be a random induced subgraph of $G$
and $p(G)$ be the probability that $G[S]$ is Hamiltonian.
Suppose for contradiction that $p(G) < p_n$.
We will show that there is a cut $(\wt{A},\wt{B})$ of $G$
with $|\wt{A}|=n+1$, $|\wt{B}|=n-1$ and $|G[\wt{B}]|=0$; call this an \emph{extremal} cut.
Then regularity will imply that $G[\wt{A},\wt{B}]$ is complete bipartite
and $G[\wt{A}]$ is a $2$-factor, i.e.~$G \in \mc{G}_n$,
contradicting the definition of $p_n$.

As in the proof of \Cref{thm:main-} and \Cref{thm:main},
it suffices to consider the near-bipartite case of \Cref{lem:casesdirac}:
there is some $A \sub V(G)$ with $n \le|A|\le(1+32\eps)n$
such that $G[A,\ov{A}]$ has at least $(1-56\eps)n^{2}$
edges and minimum degree at least $\gamma n$, where if  $|A|>n$ then
 $G[A]$ has maximum degree at most $\gamma n$. Set $B:=\ov{A}$.

 As before, whp $G[S]$ satisfies all conditions
 of \Cref{lem:bipartitehamilton}, with the possible exception of 
 the event $E$ that $(S \cap A, S \cap B)$ is a good cut of $G[S]$.
 Again we consider two cases according to the value of $k := |A|-n$. We start by ruling out the case $k > \dD\sqrt{n}$.
To see that this is impossible, note that $G[A]$ has a matching of size $\tfrac{k}{4\gG}$ by \Cref{rem:simple}, and that $p(G)$ is at least the probability that $G[A\cap S]$ has a matching of size $\tfrac{k}{20\gG}$ while at the same time $0\leq|A\cap S|-|B\cap S|\leq \tfrac{k}{20\gG}$. Thus we have $p(G) \ge \mb{P}(0 \le B(|A|,1/2) - B(|B|,1/2) \le \tfrac{k}{20\gG}) - e^{-\Theta(\sqrt{n})}$ by Chernoff.
Then by  \Cref{lem:compare} and since $f_n(t)$ is decreasing with $t$, we have
\begin{align*}
\mb{P}(0 \le B(|A|,1/2) - B(|B|,1/2) & \le \tfrac{k}{20\gG})
= 1-f_n(k+1)-f_n( \tfrac{k}{20\gG}-k+1) \\
& \ge 1/2 + \tfrac{\dD}{\sqrt{\pi}}  - O(\delta^3) - e^{-\Theta(\dD/\gG)^2}
\ge 1/2 + \dD/2,
\end{align*}
as $\gG \ll \dD$, so $p(G)>p_n$, contradiction.
Thus we can assume $k \le \dD\sqrt{n}$.
 
We will consider below a balanced cut $(A^*,B^*)$ obtained by moving $k$ vertices from $A$ to $B$
and minimum vertex covers $A',B'$ of $G[A^*],G[B^*]$. 
As before, we consider the event $E'$ that $(S \cap A^*, S \cap B^*)$ is a $k$-good cut of $G[S]$,
which implies $E$, and to conclude the proof it suffices to show $\mb{P}[E'] \ge p_n$.
Again we can assume  $\min\{\aA,\bB\} \leq \eta$, otherwise this holds by \Cref{lem:half+}.

We consider the two cases, depending on whether $|A|=|B|$ or not.

\emph{Case A:} $|A|>|B|$

As mentioned above, we consider any balanced cut $(A^*,B^*)$ obtained by moving $k\geq 1$ vertices from $A$ to $B$. Let $A',B'$ be minimum vertex covers of $G[A^*],G[B^*]$ of sizes $a = \aA\sqrt{n}, b = \bB\sqrt{n}$. 
By \Cref{lem:balancedcutbigmatching}, we have $(\aA\sqrt{n}+1)( \bB\sqrt{n}+1)\geq n+1$.

Let $M^*_A, M^*_B$ be maximal matchings in $G[A^*]$, $G[B^*]$.
Then $|M^*_A| \ge a/2$ and $|M^*_B| \ge \bcl{b/2}$.

Let $M_A, M_B$ be maximal matchings in $G[A]$, $G[B]$.
Then $|M_A| \ge |M^*_A| \ge a/2$ and $|M_B| \ge |M^*_B|-k \ge \bcl{b/2}-k$.
As noted above, we can assume $\min\{\aA,\bB\} \leq \eta$.
 
\emph{Case A1:} Suppose $\bB \le \eta$.

Then $a \ge \tfrac12 \eta^{-1} \sqrt{n}$.
By Chernoff bounds, with probability $1-e^{-\Theta(\sqrt{n})}$
there is a matching of size $a/9$ in $G[A \cap S]$.
Next we consider the choice of the random set $S' := V(M_B) \cap S \sub B \cap S$.
The expected number of surviving edges of $M_B$ is $\mb{E}[|M_B[S']|]=|M_B|/4$, so by averaging $\mb{P}(|M_B[S']| \ge b') \ge 1/8$,
where $b' := \max \{ \bcl{ (\bcl{b/2}-k)/8 }, 0 \}$.

Writing $B^\circ = B \sm V(M_B)$, to estimate $p(G)$ we consider two events 
depending on which of $|A\cap S|$ or $|B\cap S|$ is larger, giving
\begin{align*}\label{eq:prob}
 p(G) &\ge \mb{P}[ 0 \le |A \cap S|-|B \cap S|  \le a/9 ] - e^{-\Theta(\sqrt{n})}\\
&+ \tfrac18 \mb{E}_{S'} \mb{P}[ -b'  \le |A \cap S|-|S'|-|B^\circ \cap S| \le -1] ,
\end{align*}
where $|A \cap S| \sim B(n+k,1/2)$,  $|B \cap S| \sim B(n-k,1/2)$,
and $|B^\circ \cap S| \sim B(n-k-2|M_B|,1/2)$.

We estimate the first term in the bound on $p(G)$ by \Cref{lem:compare}
with parameters $(n,k,k_1,k_2) = (n,k,0,a/9)$, to get that it is larger than  $1 - (1/2 - \tfrac{k+1/2}{\sqrt{n\pi}} )-(2k^3+1)/n^{3/2} - e^{-\alpha^2/200}$.

For the second,
as $|B^\circ \cap S|-|A \cap S|+n+k \sim B(2n-2|M_B|,1/2)$,
writing  $\mu = n-|M_B|$, we have 
\begin{align*}
& \mb{P}(-b'  \le |A \cap S|-|S'|-|B^\circ \cap S| \le -1) \\
& = \mb{P}[n+k-|S'|+1 \le  B(2n-2|M_B|,1/2) \le n+k-|S'|+b']\\
& = f_\mu( k+|M_B|-|S'|+1) - f_\mu( k+|M_B|-|S'|+b'+1).
\end{align*}

Using $b',|S'| \le |M_B| \leq \eta\sqrt{n}$ where $\eta\ll 1$, and $b'\geq |M_B|/8$, 
and also $k\leq \delta \sqrt{n}$, so that $(k+|M_B|)^2/n = o((k+b')/\sqrt{n}$ for  $\eta$ sufficiently small,
by \Cref{lem:pn} we deduce
\begin{align} 
p(G)-p_n & = 1 - (1/2 - \tfrac{k+1/2}{\sqrt{n\pi}}  ) - e^{-\aA^2/200} 
+ 1_{b'>0} \tfrac18  \tfrac{b'}{\sqrt{n\pi}} 
- ( \tfrac12 + \tfrac{3/2}{\sqrt{n\pi}} ) -O\Big(\frac{(k+|M_B|)^3}{n^{3/2}}\Big) \notag\\
& \ge  \tfrac{k-1 +1_{b'>0} b'/8 }{\sqrt{n\pi}} - e^{-\aA^2/200} -o\Big(\frac{k+b'}{\sqrt{n}}\Big).
\end{align}
Recalling that $(\aA\sqrt{n}+1)( \bB\sqrt{n}+1)\geq n+1$, where $\bB \le \eta \ll 1$,
 and also that $b'\geq (\bcl{b/2}-k)/8\geq \sqrt{n}/(40\alpha)-k/8$,
we deduce $p(G)-p_n \ge \Omega(n^{-1/2})>0$ unless $k=1$ and $b'=0$. 

To conclude this subcase, it suffices to show that $b = 1$.
Indeed, as $|{A}|=n+1$ and $|{B}|=n-1$,  this would imply that $G[B^*]$ is a star, 
whose root must be the vertex $v$ moved from $A$ to $B$
when creating $(A^*,B^*)$, as $v$ has at least $\gamma n$ neighbours in $B$. Thus $G[B]$ is an independent set, so $({A},{B})$ is an extremal cut, which is a contradiction. 

As $b'=0$ we have $b\leq 2$, so it remains to rule out the possibility $b=2$, because $b\neq 0$ as there is at least one edge in $G[B^*]$.
Now, $b=2$ implies that $G[B^*]$ is either a union of a triangle and isolated vertices 
or has a matching of size $2$. The former case violates the regularity of $G$,
whereas in the latter case $\mb{P}(|M_B[S']|\geq 1) \ge 1/16$,
so a similar calculation to above (with $b'=1$ defined in the beginning of the proof, instead of $b'=0$) gives the contradiction $p(G)-p_n \ge \Omega(n^{-1/2})>0$.

\emph{Case A2:} Suppose $\aA \le \eta$.

Then $b \ge \tfrac12 \eta^{-1} \sqrt{n}$.
By Chernoff bounds, with probability $1-e^{-\Theta(\sqrt{n})}$
there is a matching of size $b/9$ in $G[B \cap S]$.
We obtain $(A^*,B^*)$ by moving $k$ vertices 
from $A$ to $B$ to maximise $e(A^*)$. Then by averaging 
$e(A^*) \ge (1-2k/n)e(A)\ge (1-2k/n)(k+1)(n+k)/2 \ge (k+1)n/2 - k^2$, as $e(A)$ has minimum degree at least $k+1$.
On the other hand, $e(A^*)$ has maximum degree at most $\gamma n$,
so its minimum vertex cover has size $a \ge e(A^*)/2\gamma n \ge (k+1)/5\gamma$.

The remainder of this case is similar to Case A1; we give the details for completeness.

We consider the choice of the random set $S' := V(M_A) \cap S \sub A \cap S$.
We have $\mb{E}[|M_A[S']|]=|M_A|/4 \ge a/8$,  
so $\mb{P}(|M_B[A]| \ge a') \ge 1/8$,
where $a'=\bcl{a/16} \ge (k+1)/90\gamma$.
We write $A^\circ = A \sm V(A_B)$ and estimate
\[ p(G) \ge \mb{P}[ 0 \le |B \cap S|-|A \cap S|  \le b/9 ] - e^{-\Theta(\sqrt{n})}
+ \tfrac18 \mb{E}_{S'} \mb{P}[ 1  \le |A^\circ \cap S|+|S'|-|B \cap S| \le a'] ,\]
where $|A \cap S| \sim B(n+k,1/2)$,  $|B \cap S| \sim B(n-k,1/2)$,
and $|A^\circ \cap S| \sim B(n+k-2|M_A|,1/2)$.

As $|A^\circ \cap S|-|B \cap S|-(n-k) \sim B(2n-2|M_A|,1/2)$,
writing  $\mu = n-|M_A|$, as in Case A1 we have 
\[\mb{P}(1  \le |A^\circ \cap S|+|S'|-|B \cap S| \le a') 
= f_\mu( k+|M_A|-|S'|+1) - f_\mu( k+|M_A|-|S'|+a'+1).\]
Since $k\leq a'\leq |M_A|\leq 100\eta \sqrt{n}$ where $\eta\ll 1$, as before we deduce
\begin{align*} 
p(G)-p_n & = 1 - (1/2 - \tfrac{k+1/2}{\sqrt{n\pi}}  ) - e^{-\bB^2/200} 
+  \tfrac18  \tfrac{a'}{\sqrt{n\pi}} 
- ( \tfrac12 + \tfrac{3/2}{\sqrt{n\pi}} ) - O\Big(\frac{|M_A|^3}{n^{3/2}}\Big) \\
& \ge  \tfrac{k-1 + a'/8 }{\sqrt{n\pi}} - e^{-\bB^2/200} +o\Big(\frac{a'}{\sqrt{n}}\Big).
\end{align*}
Recalling that $(\aA\sqrt{n}+1)( \bB\sqrt{n}+1)\geq n+1$,
we deduce $p(G)-p_n \ge \Omega(n^{-1/2}) > 0$ for any $k \ge 0$,
using $a' \ge (k+1)/90\gamma$ and $\gamma \ll 1$. 
Thus we obtain a contradiction in this case.

\emph{Case B:}  $|A|=|B|$

By symmetry we can suppose that $\beta\leq \eta$.
If $G[B]$ contains a vertex with degree at least $\gamma n$, then the proof is identical to Case A1 with $A:=A\cup\{v\}$ and $B:=B\setminus\{v\}$. Otherwise, since $G[B]$ spans at least $n/2$ edges because of regularity, it contains a matching of size $1/5\gamma$. Hence we can again follow Case A1 until \Cref{eq:prob}, where we will have $k=0$ and $b'\geq 1/50\gamma$, so clearly $p(G)-p_n>0$ also in this case.
 \end{proof}

\section{Concluding remarks} \label{sec:rem}

One question left unanswered by our paper is to determine which graph(s) $G \in \mc{G}_n$ achieve $p(G)=p_n$.
The proof of \Cref{lem:pn} shows that $|p(G)-p(G')| < e^{-\Theta(n)}$ for all $G,G' \in \mc{G}_n$,
so this seems to be a delicate question about large deviation rate functions.
A natural guess would be that $p(G)$ is minimised 
when the number of independent sets in $A$ is maximised, 
i.e.~the optimal $2$-factor should be a $C_4$-factor (see \cite{zhao2017extremal}).

Another natural direction would be to generalise the parameters of our problem, namely 
(a) the assumed degree of regularity in the graph $G$ and (b) the distribution on its induced subgraphs.
Concretely, one may consider a $d$-regular graph on $m$ vertices 
and ask about the probability $h(G,p)$ that $G[p]$ is Hamiltonian,
where $G[p]$ denotes a random induced subgraph 
with each vertex included independently with probability $p$.
If we stick to $(m,d)=(2n,n+1)$ and decrease $p$ then this puts the spotlight on 
a competing construction which we saw lurking in the background throughout the paper:
suppose (for convenience) that $n=k^2$ is a perfect square, fix a copy of $K_{n,n}$,
add $k$ vertex-disjoint $k$-vertex stars spanning each part, 
delete crossing edges between the star centres to make an $(n+1)$-regular graph.
When $G$ is this construction we saw that $h(G,1/2) \geq 0.52$ (see \Cref{fig:example}), 
so it is not a very serious competitor with $\mc{G}_n$ when $p=1/2$,
but for smaller $p$ there may be a phase transition
at which the competing construction takes over as the optimum.
Furthermore, $h(G,p)\to 0$ as $p\to 0$, 
as the largest linear forest in either part of $G[p]$ has $O(p\sqrt{n})$ edges, 
whereas $G[p]$ almost surely picks $\Omega(\sqrt{pn})$ more elements from one of the two parts of $G$.

On the other hand, if we stick to $p=1/2$ and decrease $d$, 
then to obtain a sensible question one should add an additional assumption,
such as $k$-connectivity for some $k$, otherwise $G$ may be disconnected,
in which case $h(G,1/2)$ decays exponentially in $m=|V(G)|$.
Alternatively, one could simply assume that $G$ is Hamiltonian,
as suggested to us by Alex Scott (personal communication).
We conjecture that the extremal examples $G$ for such questions 
are essentially disjoint unions of bipartite graphs,
with a few edges added to ensure the connectivity or Hamiltonicity assumption.
A more explicit and weaker form of this conjecture, which still seems interesting,
would be to show that if $d=cm$ for fixed $c \in (0,1/2)$ and $m$ large then 
$p(G) = \Omega(m^{-k/2})$ where $k = \lfloor (2c)^{-1}   \rfloor$.

One may also consider other classical combinatorial theorems
and ask for robust analogues in the sense of this paper.
For example, consider the Hajnal-Szemer\'edi Theorem \cite{HSz} (see also \cite{kierstead2008short})
that any $n$-vertex graph with $\delta(G)\geq (r-1)n/r$ (where $r$ divides $n$) contains a $K_r$-factor, 
i.e.~a partition of $V(G)$ into sets inducing copies of $K_r$ in $G$.
By analogy with \Cref{conj:EF}, we pose the following conjecture.

\begin{conjecture} \label{conj:HS}
For any $r \ge 2$ there is some $\eps>0$
so that if $G$ is an  $(\tfrac{(r-1)n}{r}+1)$-regular graph on $n$ vertices, where $r \mid n$,
then at least $\eps 2^n$ subsets of $V(G)$ induce a $K_r$-factor.
\end{conjecture}

A plausible class of extremal constructions for \Cref{conj:HS} may be the $r$-partite analogue of $\mc{G}_n$,
i.e.~ slightly unbalanced complete $r$-partite graphs with a suitable factor in the largest part $A$.
This suggests that the optimal $\eps$ should be $1/r^2$, where in the random induced subgraph $G[S]$
we have one probability factor of $1/r$ for $r \mid |S|$ and another for $A \cap S$ being the largest part.

\noindent \textbf{Acknowledgement.} We would like to thank Zach Hunter, Teng Liu and Aleksa Milojević for carefully reading the first version of our manuscript and helping us eliminate some minor errors.
\bibliographystyle{plain}

\providecommand{\MR}[1]{}

\end{document}